\documentclass[english]{article}
\usepackage[T1]{fontenc}
\usepackage[latin9]{inputenc}
\usepackage{color}
\usepackage{refstyle}
\usepackage{enumitem}
\usepackage{amsmath}
\usepackage{amsthm}
\usepackage{amssymb}
\usepackage{setspace}
\makeatletter

\AtBeginDocument{\providecommand\secref[1]{\ref{sec:#1}}}
\AtBeginDocument{\providecommand\defref[1]{\ref{def:#1}}}
\AtBeginDocument{\providecommand\lemref[1]{\ref{lem:#1}}}
\AtBeginDocument{\providecommand\corref[1]{\ref{cor:#1}}}
\AtBeginDocument{\providecommand\thmref[1]{\ref{thm:#1}}}
\AtBeginDocument{\providecommand\remref[1]{\ref{rem:#1}}}
\AtBeginDocument{\providecommand\propref[1]{\ref{prop:#1}}}
\RS@ifundefined{subsecref}
  {\newref{subsec}{name = \RSsectxt}}
  {}
\RS@ifundefined{thmref}
  {\def\RSthmtxt{theorem~}\newref{thm}{name = \RSthmtxt}}
  {}
\RS@ifundefined{lemref}
  {\def\RSlemtxt{lemma~}\newref{lem}{name = \RSlemtxt}}
  {}

\theoremstyle{plain}
\newtheorem{thm}{\protect\theoremname}[section]
\theoremstyle{definition}
\newtheorem{defn}[thm]{\protect\definitionname}
\theoremstyle{plain}
\newtheorem{lem}[thm]{\protect\lemmaname}
\theoremstyle{remark}
\newtheorem{rem}[thm]{\protect\remarkname}
\theoremstyle{plain}
\newtheorem{prop}[thm]{\protect\propositionname}
\theoremstyle{plain}
\newtheorem{cor}[thm]{\protect\corollaryname}
\theoremstyle{definition}
\newtheorem{example}[thm]{\protect\examplename}

\@ifundefined{date}{}{\date{}}

\AtBeginDocument{\providecommand\remref[1]{\ref{rem:#1}}}
\AtBeginDocument{}
\AtBeginDocument{\providecommand\propref[1]{\ref{prop:#1}}}
\RS@ifundefined{propref}{\newref{prop}{name =Proposition~,names = Propositions~}}{}
\AtBeginDocument{\providecommand\corref[1]{\ref{cor:#1}}}
\RS@ifundefined{corref}{\newref{cor}{name =Corollary~,names = Corollaries~}}{}
\RS@ifundefined{remref}{\newref{rem}{name =Remark~,names = Remarks~}}{}
\RS@ifundefined{defref}{\newref{def}{name =Definition~,names = Definitions~}}{}
\RS@ifundefined{exaref}{\newref{exa}{name =Example~,names = Examples~}}{}
\RS@ifundefined{eqqref}{\newref{eqq}{name =Equality~,names = Equality~}}{}
\AtBeginDocument{}
\AtBeginDocument{\renewcommand\secref[1]{Section~\ref{sec:#1}}}

\def\@fnsymbol#1{\ensuremath{\ifcase#1\or *\or **\or \ddagger\or
   \mathsection\or \mathparagraph\or \|\or \dagger\dagger
   \or \ddagger\ddagger \else\@ctrerr\fi}}

\usepackage{enumitem, hyperref}
\makeatletter
\def\namedlabel#1#2{\begingroup
    #2%
    \def\@currentlabel{#2}%
    \phantomsection\label{#1}\endgroup
}

\usepackage{tikz}

\usepackage{tocbibind}
\usepackage{youngtab}
\usepackage{authblk}
\usetikzlibrary{matrix}
\usepackage{bbding}
\usepackage{ytableau}
\usepackage{amsmath}
\usepackage{geometry}
\usepackage{mathtools}

\tikzset{   pt/.style={insert path={node[scale=2]{.}}},   dnup/.style={insert path={ [pt] .. controls +(0,1) and +(0,-1) .. +(#1,2) [pt]}},   dndn/.style={insert path={ [pt] .. controls +(0,0.25) and +(0,0.25) .. +(#1,0) [pt]}},   upup/.style={insert path={ [pt] .. controls +(0,-0.25) and +(0,-0.25) .. +(#1,0) [pt]}}, upup2/.style={insert path={ [pt] .. controls +(0,-0.5) and +(0,-0.5) .. +(#1,0) [pt]}}, }

\newcommand{\Rt}{\widetilde{\mathcal{R}}_E}
\newcommand{\Lt}{\widetilde{\mathcal{L}}_E}

\DeclareMathOperator{\id}{id}

\DeclareMathOperator{\Rad}{Rad}
\DeclareMathOperator{\im}{\mathsf{im}}

\DeclareMathOperator{\PT}{\mathcal{PT}}

\DeclareMathOperator{\T}{\mathcal{T}}
\DeclareMathOperator{\B}{\mathcal{B}}
\DeclareMathOperator{\Rc}{\mathcal{R}}
\DeclareMathOperator{\Lc}{\mathcal{L}}

\DeclareMathOperator{\Jc}{\mathcal{J}}

\DeclareMathOperator{\dom}{\mathsf{dom}}

\DeclareMathOperator{\op}{op}

\DeclareMathOperator{\db}{\mathbf{d}}
\DeclareMathOperator{\rb}{\mathbf{r}}

\DeclareMathOperator{\Op}{\mathcal{O}}
\DeclareMathOperator{\OF}{\mathcal{OF}}

\def\RSlemtxt{Lemma~}
\def\RSthmtxt{Theorem~}

\usepackage{babel}
\providecommand{\corollaryname}{Corollary}
  \providecommand{\definitionname}{Definition}
  \providecommand{\examplename}{Example}
  \providecommand{\lemmaname}{Lemma}
  \providecommand{\propositionname}{Proposition}
  \providecommand{\remarkname}{Remark}
 
\providecommand{\theoremname}{Theorem}

\makeatother

\usepackage{babel}
\providecommand{\corollaryname}{Corollary}
\providecommand{\definitionname}{Definition}
\providecommand{\examplename}{Example}
\providecommand{\lemmaname}{Lemma}
\providecommand{\propositionname}{Proposition}
\providecommand{\remarkname}{Remark}
\providecommand{\theoremname}{Theorem}

\begin{document}
\title{The algebra of the monoid of order-preserving functions on an $n$-set
and other reduced $E$-Fountain semigroups}
\author{Itamar Stein\thanks{Mathematics Unit, Shamoon College of Engineering, 77245 Ashdod, Israel}\\
\Envelope \, Steinita@gmail.com}
\maketitle
\begin{abstract}
With every reduced $E$-Fountain semigroup $S$ which satisfies the
generalized right ample condition we associate a category with zero
morphisms $\mathcal{C}(S)$. Under some assumptions we prove an isomorphism
of $\Bbbk$-algebras $\Bbbk S\simeq\Bbbk_{0}\mathcal{C}(S)$ between
the semigroup algebra and the contracted category algebra where $\Bbbk$
is any commutative unital ring. This is a simultaneous generalization
of a former result of the author on reduced E-Fountain semigroups
which satisfy the congruence condition, a result of Junying Guo and
Xiaojiang Guo on strict right ample semigroups and a result of Benjamin
Steinberg on idempotent semigroups with central idempotents. The applicability
of the new isomorphism is demonstrated with two well-known monoids
which are not members of the above classes. The monoid of order-preserving
functions on an $n$-set and the monoid of binary relations with demonic
composition.
\end{abstract}

\section{Introduction}

Given a semigroup $S$ and a commutative ring $\Bbbk$, it is of interest
to study the semigroup algebra $\Bbbk S$. One approach is to associate
with $S$ a category $\mathcal{C}(S)$ such that the semigroup algebra
$\Bbbk S$ is isomorphic to the category algebra $\Bbbk\mathcal{C}(S)$.
It is usually easier to study the category algebra because many of
the natural basis elements are orthogonal. A pioneering work in this
direction was done by Solomon \cite{Solomon1967} who proved that
the semigroup algebra of a finite semilattice is isomorphic to the
algebra of a discrete category. This was generalized to finite inverse
semigroups by Steinberg \cite{Steinberg2006}. From the ESN theorem
\cite[Theorem 4.1.8]{Lawson1998} we know that every inverse semigroup
$S$ is associated with an inductive groupoid $G(S)$. Steinberg proved
that if $S$ is finite then $\Bbbk S\simeq\Bbbk G(S)$ and this result
is fundamental in the representation theory of finite inverse semigroups
\cite[Chapter 9]{Steinberg2016}. Guo and Chen \cite{Guo2012} obtained
a similar result for finite ample semigroups. Motivated by questions
on self-injective and Frobenius semigroup algebras, Guo and Guo \cite{Guo2018}
generalized this result to a class of semigroup they call strict right
ample semigroups. Taking a different direction, the author extended
the result of Guo and Chen on finite ample semigroups to a class of
right restriction $E$-Ehresmann semigroups \cite{Stein2017,Stein2018erratum}
(see also \cite{Wang2017}) and later to a wider class of reduced
$E$-Fountain semigroups \cite{Stein2021}. Let $S$ be a semigroup
and let $E$ be a distinguished subset of idempotents $E\subseteq S$
with the property that $ef=e\iff fe=e$ for $e,f\in E$. The semigroup
$S$ is called a reduced $E$-Fountain semigroup if for every element
$a\in S$ the sets of its right and left identities from $E$ have
minimum elements (with respect to the natural partial order on idempotents).
The minimal right identity of $a$ from $E$ is denoted $a^{\ast}$
and the minimal left identity of $a$ from $E$ is denoted $a^{+}$.
Reduced $E$-Fountain semigroups form a wide class that contains:
inverse semigroups, $\Jc$-trivial semigroups and many other types
of semigroups. If $S$ is reduced $E$-Fountain then we can define
a graph $\mathcal{C_{\bullet}}(S)$ whose objects are the idempotents
of $E$ and every $a\in S$ corresponds to a morphism $C(a)$ from
$a^{\ast}$ to $a^{+}$. If $S$ satisfies the congruence condition
(the identities $(ab)^{\ast}=(a^{\ast}b)^{\ast}$ and $(ab)^{+}=(ab^{+})^{+}$)
then $\mathcal{\mathcal{C_{\bullet}}}(S)$ becomes a category when
we define $C(b)C(a)=C(ba)$ if $b^{\ast}=a^{+}$. If, in addition,
$S$ satisfies some further conditions that will be described below
then it is proved in \cite[Theorem 4.4]{Stein2021} that $\Bbbk S\simeq\Bbbk\mathcal{C_{\bullet}}(S)$
for any commutative unital ring $\Bbbk$. This result has led to several
applications regarding quivers \cite{Stein2016,Stein2020}, global-dimension
\cite{Stein2019} and projective modules \cite{Margolis2021} for
algebras of semigroups of partial functions. A special case of this
result is the neat proof that the algebra of the Catalan monoid is
an incidence algebra \cite{Margolis2018A} and there are also applications
to the study of certain partition monoids \cite{East2021}.

Recently, as part of a detailed study of semigroup determinants \cite{Steinberg2022},
Steinberg proved a similar result for the class of finite idempotent
semigroups with central idempotents. These are finite semigroups that
satisfy $S^{2}=S$ and $ea=ae$ for every element $a\in S$ and idempotent
$e\in S$. Such a semigroup $S$ is reduced $E$-Fountain if we take
$E$ to be the set of all idempotents of $S$. In this case the associated
graph $\mathcal{C}_{\bullet}(S)$ contains only endomorphisms and
despite not being a category it is possible to add zero morphisms
and turn it into a category with zero morphisms. Then, the semigroup
algebra is isomorphic to the contracted algebra of the category with
zero morphisms \cite[Theorem 6.5]{Steinberg2022}.

The goal of this paper is to achieve a simultaneous generalization
for our result on reduced $E$-Fountain semigroups that satisfy the
congruence condition, Guo and Guo's result on strict right ample semigroups
and Steinberg's result on idempotent semigroups with central idempotents.
In \secref{Preliminaries} we give some required preliminaries. In
\secref{Isomorphism} we consider the generalized right ample condition,
discussed in \cite{Stein2021,stein2022} and show that it is equivalent
to the condition $(ba)^{\ast}=a^{\ast}\implies(ba^{+})^{\ast}=a^{+}$
for every $a,b\in S$. If a reduced $E$-Fountain semigroup $S$ satisfies
the generalized right ample condition, then we can add zero morphisms
to $\mathcal{C_{\bullet}}(S)$ and obtain a new graph $\mathcal{C}(S)$.
We define composition $C(b)\cdot C(a)$ of two non-zero morphisms
$C(b),C(a)$ with $b^{\ast}=a^{+}$ to be equal to $C(ba)$ if $(ba)^{\ast}=b^{\ast},\quad(ba)^{+}=a^{+}$
and $0$ otherwise. We prove that this endows $\mathcal{C}(S)$ with
a structure of a category with zero morphisms. We define a relation
$\trianglelefteq_{l}$ on $S$ by the rule that $a\trianglelefteq_{l}b$
if $a=be$ for an $e\in E$. The main theorem of \secref{Isomorphism}
shows that if a reduced $E$-Fountain semigroup $S$ satisfies the
generalized right ample condition and $\trianglelefteq_{l}$ is contained
in a principally finite partial order then $\Bbbk S\simeq\Bbbk_{0}\mathcal{C}(S)$
for every commutative unital ring $\Bbbk$ (where $\Bbbk_{0}\mathcal{C}(S)$
is the contracted algebra of $\mathcal{C}(S)$). In \secref{Generalization}
we explain why this is a generalization of the three above-mentioned
results of the author, Guo and Guo, and Steinberg. 

In \secref{Examples} we show that the above isomorphism applies to
two additional examples which are not included in any of the three
older results: The monoid $\Op_{n}$ of all order-preserving functions
on an $n$-element set (see \cite[Chapter 14]{Ganyushkin2009b} and
references therein) and the monoid of binary relations with demonic
composition, whose origins are in computer science \cite{Demonic2006}.

\textbf{Acknowledgment:} The author is grateful to the referee for
suggestions which improved the paper.

\section{\label{sec:Preliminaries}Preliminaries}

\subsection{\label{subsec:Semigroups}Semigroups and categories}

Let $S$ be a semigroup and let $S^{1}=S\cup\{1\}$ be the monoid
formed by adjoining a formal unit element. Recall that Green's preorders
$\leq_{\Rc}$, $\leq_{\Lc}$ and $\leq_{\Jc}$ are defined by:
\begin{align*}
a\,\leq_{\Rc}\,b\iff & aS^{1}\subseteq bS^{1}\\
a\,\leq_{\Lc}\,b\iff & S^{1}a\subseteq S^{1}b\\
a\,\leq_{\Jc}\,b\iff & S^{1}aS^{1}\subseteq S^{1}bS^{1}.
\end{align*}
The associated Green's equivalence relations on $S$ are denoted by
$\Rc$, $\Lc$ and $\Jc$. Additional basic notions and proofs of
semigroup theory can be found in standard textbooks such as \cite{Howie1995}.

We denote by $E(S)$ the set of idempotents of the semigroup $S$.
The natural partial order on $E(S)$ is defined by $e\leq f\iff(ef=e=fe)$
and note that $\leq\,=\,\leq_{\Rc}\cap\leq_{\Lc}$. For every subset
of idempotents $E\subseteq E(S)$ and every $a\in S$ we denote by
$a_{E}$ the set of right identities of $a$ from $E$:
\[
a_{E}=\{e\in E\mid ae=a\}
\]
Dually $\prescript{}{E}{a}$ is the set of left identities from $E$.
We define two equivalence relations $\Lt$ and $\Rt$ on $S$ by 
\[
a\,\Lt\,b\iff(a_{E}=b_{E})
\]
\[
a\,\Rt\,b\iff(\ensuremath{\prescript{}{E}{a}}=\ensuremath{\prescript{}{E}{b}}).
\]

These relations are one type of ``generalized'' Green's relations
(see \cite{Gould2010}). Note that $\Lc\subseteq\Lt$ and $\Rc\subseteq\Rt$.
\begin{defn}
The semigroup $S$ is called \emph{$E$-Fountain} if every $\Lt$-class
contains an idempotent from $E$ and every $\Rt$-class contains an
idempotent from $E$. 
\end{defn}

We remark that this property is also called ``$E$-semiabundant''
in the literature.
\begin{defn}
An $E$-Fountain semigroup $S$ is called \emph{reduced} if 
\[
\mbox{\ensuremath{ef=e\iff fe=e}}
\]
 for every $e,f\in E$. Equivalently, $S$ is reduced if $\leq_{\Lc}=\leq_{\Rc}=\leq$
when restricted to the set $E$.
\end{defn}

A reduced \emph{$E$-}Fountain\emph{ }semigroup is called a ``DR
semigroup'' in \cite{Stokes2015}.

Let $S$ be a semigroup and let $E\subseteq S$ be a subset of idempotents
which satisfies $\mbox{\ensuremath{ef=e\iff fe=e}}$ for every $e,f\in E$.
Then every $\Lt$-class ($\Rt$-class) contains at most one idempotent
from $E$. The $\Lt$-class ($\Rt$-class) of $a\in S$ contains an
element of $E$ if and only if the set $a_{E}$ (respectively, $\ensuremath{\prescript{}{E}{a}}$)
contains a minimum element. 

If $S$ is a reduced $E$-Fountain semigroup, the unique idempotent
from $E$ in the $\Lt$-class ($\Rt$-class) of $a\in S$ is denoted
$a^{\ast}$ (respectively, $a^{+}$). We emphasize that if $e\in E$
satisfies $ae=a$ then $ea^{\ast}=a^{\ast}e=a^{\ast}$ and dually
if $e\in E$ satisfies $ea=a$ then $ea^{+}=a^{+}e=a^{+}$. See \cite{Stokes2015}
for proofs and additional details. 

The following simple observation will be very useful.
\begin{lem}
\label{lem:AstPlusLemma}Let $S$ be a reduced $E$-Fountain semigroup
and let $a,b\in S$. Then $(ba)^{\ast}\leq a^{\ast}$ and $(ba)^{+}\leq b^{+}$.
\end{lem}

\begin{proof}
The equality $baa^{\ast}=ba$ implies $(ba)^{\ast}a^{\ast}=(ba)^{\ast}$
so $(ba)^{\ast}\leq a^{\ast}$. Proving $(ba)^{+}\leq b^{+}$ is dual.
\end{proof}
Green's relation $\Lc$ ($\Rc$) is always a right congruence (respectively,
left congruence), but this is not the case with the generalized Green's
relations $\Lt$ and $\Rt$.
\begin{defn}
Let $S$ be a reduced $E$-Fountain semigroup. We say that $S$ satisfies
the \emph{right} \emph{congruence condition} if $\Lt$ is a right
congruence. It is known that this is equivalent to the identity $(ab)^{\ast}=(a^{\ast}b)^{\ast}$
for every $a,b\in S$ \cite[Lemma 4.1]{Gould2010}. Dually, $S$ satisfies
the \emph{left} \emph{congruence condition} if $\Rt$ is a left congruence
which is equivalent to the identity $(ab)^{+}=(ab^{+})^{+}$.
\end{defn}

Denote by $\PT_{X}$ the monoid of all partial functions on a set
$X$ - composing functions from right to left. Then $\PT_{X}^{\op}$
is the monoid with the same underlying set but composing functions
from left to right. Let $\psi:S\to\PT_{X}$ be a semigroup homomorphism.
We follow \cite[Definition 3.2]{Gould2009} and call $\psi$ an \emph{incomplete
left action} of $S$ on $X$. Equivalently, we say that $X$ is an
\emph{incomplete left $S$-action} (note that \emph{partial }action
usually has a different meaning - for instance, see \cite[Definition 2.2]{Kudryavtseva2023}).
Another convention is to write $s\bullet x$ instead of $\psi(s)(x)$
for $s\in S$ and $x\in X$. Note that $s\bullet x$ might be undefined.
Of course, we could have added a symbol $0$ and talk about actions
of $S$ on $X\cup\{0\}$ such that $s\bullet0=s$ for every $s\in S$.
However, we prefer not to add $0$ at this point.

It is known \cite[Lemma 3.4]{stein2022} that every $\Lc$-class $L$
is an incomplete left $S$-action according to 
\[
s\bullet x=\begin{cases}
sx & sx\in L\\
\text{undefined} & \text{otherwise}
\end{cases}
\]
for every $x\in L$ and $s\in S$. The same is true if we replace
$\Lc$-class by $\Lt$-class \cite[Lemma 5.1]{Margolis2021}. Assume
that $X,Y$ are incomplete left $S$-actions. A partial function $f:X\to Y$
is a \emph{partial} \emph{homomorphism of incomplete left $S$-actions}
if for every $x\in X$ and $s\in S$, $f(s\bullet x)$ is defined
if and only if $s\bullet f(x)$ is defined and in this case 
\[
f(s\bullet x)=s\bullet f(x).
\]

There is a dual notion of course. An\emph{ incomplete right action}
of $S$ on a set $X$ is a homomorphism $\mbox{\ensuremath{\psi:S\to\PT_{X}^{\op}}}$.
Equivalently, we say that $X$ is an \emph{incomplete right $S$-action}.
Partial homomorphisms of incomplete right $S$-actions are defined
in the obvious dual way.

We can regard any graph $\mathcal{C}$ as set of objects, denoted
$\mathcal{C}^{0}$, and a set of morphisms, denoted $\mathcal{C}^{1}$.
It is equipped with two functions $\db,\rb:\mathcal{C}^{1}\to\mathcal{C}^{0}$
associating every morphism $m\in\mathcal{C}^{1}$ with its \emph{domain
$\db(m)$ }and its \emph{range} $\rb(m)$. The set of morphisms with
domain $e$ and range $f$ (also called an \emph{hom-set}) is denoted
$\mathcal{C}(e,f)$. Since any category has an underlying graph, the
above notations hold also for categories. We compose morphisms in
a category from right to left so $m^{\prime}m$ is defined in a category
$\mathcal{C}$ if $\db(m^{\prime})=\rb(m)$.

A category $C$ is called a \emph{category with zero morphisms} if
for every $x,y\in\mathcal{C}^{0}$ there exists a zero morphism $0_{y,x}\in\mathcal{C}(x,y)$
and for every three objects $e,f,g\in\mathcal{C}^{0}$ and morphisms
$m^{\prime}\in\mathcal{C}(g,f)$, $m\in\mathcal{C}(e,g)$ we have
\[
m^{\prime}0_{g,e}=0_{f,e}=0_{f,g}m.
\]
Equivalently, a category with zero morphisms is a category enriched
over the category of pointed sets.

\subsection{Algebras}

Let $\Bbbk$ be a unital commutative ring and let $S$ be a semigroup.
The\emph{ semigroup algebra} $\mathbb{\Bbbk}S$ is defined in the
following way. It is a free $\mathbb{\Bbbk}$-module with basis the
elements of $S$, that is, it consists of all formal linear combinations
\[
\{k_{1}s_{1}+\ldots+k_{n}s_{n}\mid k_{i}\in\mathbb{\Bbbk},\,s_{i}\in S\}.
\]
The multiplication in $\mathbb{\Bbbk}S$ is the linear extension of
the semigroup multiplication. Note that in general $\Bbbk S$ might
not possess a unit element. 

The \emph{category algebra} $\mathbb{\Bbbk}\mathcal{C}$ of a (small)
category $\mathcal{C}$ is defined in the following way. It is a free
$\mathbb{\Bbbk}$-module with basis the morphisms of $\mathcal{C}$,
that is, it consists of all formal linear combinations
\[
\{k_{1}m_{1}+\ldots+k_{n}m_{n}\mid k_{i}\in\mathbb{\Bbbk},\,m_{i}\in\mathcal{C}^{1}\}.
\]
The multiplication in $\mathbb{\Bbbk}\mathcal{C}$ is the linear extension
of the following:
\[
m^{\prime}\cdot m=\begin{cases}
m^{\prime}m & \text{if \ensuremath{m^{\prime}m} is defined}\\
0 & \text{otherwise}.
\end{cases}
\]

If $\mathcal{C}$ has a finite number of objects then the unit element
of $\Bbbk\mathcal{C}$ is${\displaystyle \sum_{e\in\mathcal{C}^{0}}1_{e}}$
where $1_{e}$ is the identity morphism of the object $e$ of $\mathcal{C}$.

Let $\mathcal{C}$ be a category with zero morphisms and denote by
$Z$ the ideal of $\Bbbk\mathcal{C}$ generated by the zero morphisms.
The \emph{contracted category algebra }of $\mathcal{C}$ is defined
by $\Bbbk_{0}\mathcal{C}=\Bbbk\mathcal{C}/Z$. In a contracted category
algebra we identify all the zero morphisms with the zero element of
the algebra.

Let $A$ be a finite dimensional and unital $\Bbbk$-algebra. Let
$E=\{e_{1},\ldots,e_{n}\}$ be a set of idempotents from $A$. Recall
that $E$ is a complete set of orthogonal idempotents if ${\displaystyle \sum_{i=1}^{n}e_{i}=1_{A}}$
and $e_{i}e_{j}=0$ if $i\neq j$. If, in addition, the idempotents
of $E$ are \emph{central} (i.e. $ea=ae$ for every $e\in E$ and
$a\in A$) then $A\simeq{\displaystyle \prod_{i=1}^{n}e_{i}Ae_{i}}$
is an isomorphism of $\Bbbk$-algebras. 

Two $\Bbbk$-algebras $A$ and $B$ are \emph{Morita equivalent }if
the category of $A$-modules is equivalent to the category of $B$-modules.

\textcolor{black}{Assume that $\Bbbk$ is a field and recall that
a $\Bbbk$-algebra }$A$\textcolor{black}{{} is called }\textcolor{black}{\emph{semisimple}}\textcolor{black}{{}
if any }$A$\textcolor{black}{-module is semisimple (= a direct sum
of simple modules). We denote by }$\Rad A$\textcolor{black}{{} the
}\textcolor{black}{\emph{radical}}\textcolor{black}{{} of }$A$\textcolor{black}{,
which is the minimal ideal such that }$A/\Rad A$\textcolor{black}{{}
is a semisimple algebra. The radical }$\Rad A$\textcolor{black}{{}
is also the only nilpotent ideal }$I$\textcolor{black}{{} of }$A$\textcolor{black}{{}
with the property that }$A/I$\textcolor{black}{{} is a semisimple }$\Bbbk$\textcolor{black}{-algebra.}

\section{\label{sec:Isomorphism}An isomorphism theorem and the generalized
right ample condition}
\begin{defn}
\label{def:Generalized_right_ample_id}Let $S$ be a reduced $E$-Fountain
semigroup. We say that $S$ satisfies the \emph{generalized right
ample condition} (or the \emph{generalized right ample identity})
if 
\[
(b(a(ba)^{\ast})^{+})^{\ast}=(a(ba)^{\ast})^{+}
\]
for every $a,b\in S$.
\end{defn}

\begin{rem}
This identity is given in \cite{stein2022} in the form $(b^{\ast}(a(b^{\ast}a)^{\ast})^{+})^{\ast}=(a(b^{\ast}a)^{\ast})^{+}$
for $a,b\in S$ (or $(e(a(ea)^{\ast})^{+})^{\ast}=(a(ea)^{\ast})^{+}$
for $e\in E$ and $a\in S$). This is equivalent to the identity in
\defref{Generalized_right_ample_id} if we assume the right congruence
condition. In this paper we do not assume the right congruence condition
so we are forced to take a slightly different form of the identity.
\end{rem}

Let $S$ be a reduced $E$-Fountain semigroup. We denote the $\Lt$-class
of $e\in E$ by $\Lt(e)$. We have already mentioned in the preliminaries
that for every $e\in E$, the set $\Lt(e)$ is an incomplete $S$-action
according to
\[
s\bullet x=\begin{cases}
sx & sx\in\Lt(e)\\
\text{undefined} & \text{otherwise}
\end{cases}
\]

for $s\in S$ and $x\in\Lt(e)$. 

For every $a\in S$ we can define a partial function $r_{a}:\Lt(a^{+})\to\Lt(a^{\ast})$
by 
\[
r_{a}(b)=\begin{cases}
ba & (ba)^{\ast}=a^{\ast}\\
\text{undefined} & \text{otherwise}
\end{cases}
\]
where $b\in\Lt(a^{+})$.

The following proposition strengthens \cite[Theorem 3.6]{stein2022}.
\begin{prop}
\label{prop:Equivalent_characterizations_of_gra}Let $S$ be a reduced
$E$-Fountain semigroup. The following conditions are equivalent.
\begin{enumerate}
\item The semigroup $S$ satisfies the generalized right ample condition.
\item The semigroup $S$ satisfies $(ba)^{\ast}=a^{\ast}\implies(ba^{+})^{\ast}=a^{+}$
for every $a,b\in S$.
\item The partial function $r_{a}:\Lt(a^{+})\to\Lt(a^{\ast})$ is a partial
homomorphism of incomplete left $S$-actions for every $a\in S$.
\end{enumerate}
\end{prop}

\begin{proof}
For $(1\Rightarrow2)$ substitute $(ba)^{\ast}=a^{\ast}$ in the generalized
right ample identity to obtain 
\[
\mbox{\ensuremath{(b(aa^{\ast})^{+})^{\ast}=(aa^{\ast})^{+}}}
\]
 so we get $(ba^{+})^{\ast}=a^{+}$ as required. For $(2\Rightarrow1)$,
assume that $(yx)^{\ast}=x^{\ast}$ implies $(yx^{+})^{\ast}=x^{+}$.
Note that by \lemref{AstPlusLemma}
\[
(ba)^{\ast}=(ba(ba)^{\ast})^{\ast}\leq(a(ba)^{\ast})^{\ast}\leq(ba)^{\ast}
\]
so $(a(ba)^{\ast})^{\ast}=(ba)^{\ast}$ and therefore, 
\[
(b(a(ba)^{\ast}))^{\ast}=(ba)^{\ast}=(a(ba)^{\ast})^{\ast}.
\]
Setting $y=b$ and $x=a(ba)^{\ast}$ implies that 
\[
(b(a(ba)^{\ast})^{+})^{\ast}=(a(ba)^{\ast})^{+}
\]
as required. For $(3\Rightarrow2)$ we assume that $(ba)^{\ast}=a^{\ast}$.
This means that $b\bullet r_{a}(a^{+})=ba\in\Lt(a^{\ast})$. Since
$r_{a}$ is a partial homomorphism of incomplete left $S$-actions
we know that $r_{a}(b\bullet a^{+})$ is also defined so in particular
$ba^{+}\in\Lt(a^{+})$ hence $(ba^{+})^{\ast}=a^{+}$. It is left
to prove $(2\Rightarrow3)$. Assume that $(ba)^{\ast}=a^{\ast}$ implies
$(ba^{+})^{\ast}=a^{+}$ and let $s,a\in S$ and $x\in\Lt(a^{+})$.
If $s\bullet r_{a}(x)$ is defined then $(xa)^{\ast}=a^{\ast}$ and
$(sxa)^{\ast}=a^{\ast}$. According to the assumption (with $b=sx$)
we find that $(sxa^{+})^{\ast}=a^{+}$ but $a^{+}=x^{\ast}$ so $(sx)^{\ast}=x^{\ast}$.
This shows that $sx\in\Lt(a^{+})$ and therefore $s\bullet x$ is
defined. Moreover $(sxa)^{\ast}=a^{\ast}$ so $r_{a}(s\bullet x)$
is defined. In the other direction if $r_{a}(s\bullet x)$ is defined
then $(sx)^{\ast}=x^{\ast}$ and $(sxa)^{\ast}=a^{\ast}$. Since $a^{\ast}=(sxa)^{\ast}\leq(xa)^{\ast}\leq a^{\ast}$
we see that $(xa)^{\ast}=a^{\ast}$ so $r_{a}(x)$ is defined and
$(sxa)^{\ast}=a^{\ast}$ shows that $s\bullet r_{a}(x)$ is defined
as well. This finally proves that $s\bullet r_{a}(x)$ is defined
if and only if $r_{a}(s\bullet x)$ is defined. In this case it is
clear that both equal $sxa$ so we are done.
\end{proof}
\begin{defn}
Let $S$ be a reduced $E$-Fountain semigroup. We associate with $S$
a graph $\mathcal{C}_{\bullet}(S)$ as follows: The objects correspond
to the set of idempotents $E$ and for every $a\in S$ we associate
a morphism $C(a)$ with $\db(C(a))=a^{\ast}$ and $\rb(C(a))=a^{+}$.
\end{defn}

It is tempting to define $C(b)C(a)=C(ba)$ whenever $\db(C(b))=\rb(C(a))$
(or equivalently, $b^{\ast}=a^{+}$) in hope that this will give us
a category. The problem is that we do not know if $b^{\ast}=a^{+}$
implies $\db(C(ba))=\db(C(a))$ and $\rb(C(ba))=\rb(C(b))$ (that
is, $(ba)^{\ast}=a^{\ast}$ and $(ba)^{+}=b^{+}$). We refer the reader
to \cite{Stokes2023} for a detailed study of the cases where we do
obtain a category in this way. What we can prove is that $(ba)^{\ast}=a^{\ast}$
implies $(ba)^{+}=b^{+}$ if $S$ satisfies the generalized right
ample condition.
\begin{lem}
\label{lem:Ast_implies_plus}Let $S$ be a reduced $E$-Fountain semigroup
which satisfies the generalized right ample condition. If $a,b\in S$
are such that $b^{\ast}=a^{+}$ and $(ba)^{\ast}=a^{\ast}$ then $(ba)^{+}=b^{+}$.
\end{lem}

\begin{proof}
First observe that
\[
((ba)^{+}ba)^{\ast}=(ba)^{\ast}=a^{\ast}.
\]
The generalized right ample condition implies that 
\[
((ba)^{+}b)^{\ast}=((ba)^{+}bb^{\ast})^{\ast}=((ba)^{+}ba^{+})^{\ast}=a^{+}=b^{\ast}.
\]
Another use of the generalized right ample condition will give 
\[
((ba)^{+}b^{+})^{\ast}=b^{+}.
\]
By \lemref{AstPlusLemma}, $(ba)^{+}\leq b^{+}$ so 
\[
(ba)^{+}=((ba)^{+})^{\ast}=((ba)^{+}b^{+})^{\ast}=b^{+}
\]
as required.
\end{proof}
\begin{cor}
\label{cor:The_right_congruence_case}Let $S$ be a reduced $E$-Fountain
semigroup which satisfies the right congruence identity and the generalized
right ample condition. Then, defining $C(b)\cdot C(a)=C(ba)$ when
$b^{\ast}=a^{+}$ endows the graph $\mathcal{C}_{\bullet}(S)$ with
a category structure.
\end{cor}

\begin{proof}
If $b^{\ast}=a^{+}$ then the right congruence condition implies $(ba)^{\ast}=(b^{\ast}a)^{\ast}=(a^{+}a)^{\ast}=a^{\ast}$.
The rest follows from \lemref{Ast_implies_plus}.
\end{proof}
If $S$ is not right congruence then $\mathcal{C}_{\bullet}(S)$ is
not a category in general. What we can say is that if the generalized
right ample condition holds then we can add zero morphisms and make
it a category with zero morphisms. We define a new graph $\mathcal{C}(S)$
by adding zero morphisms. For every $e,f\in\mathcal{C}_{\bullet}^{0}(S)$
a add a morphism $0_{f,e}$ with domain $e$ and range $f$. It will
be convenient to drop the subscript and write just $0$ since the
domain and range can be understood from the context. Note that only
the non-zero morphisms of $\mathcal{C}(S)$ are in one-to-one correspondence
with elements of $S$. 
\begin{defn}
\label{def:Definition_of_composition}Define composition in $C(S)$
as follows: For two morphisms $m^{\prime}$ and $m$, with $\mbox{\ensuremath{\db(m^{\prime})=\rb(m)}}$,
we define $m^{\prime}\cdot m=0$ if at least one of $m^{\prime}$
or $m$ is a zero morphism. Otherwise, we can write $m^{\prime}=C(b)$
and $m=C(a)$ where $b,a\in S$ with $b^{\ast}=a^{+}$ and in this
case: 
\[
C(b)\cdot C(a)=\begin{cases}
C(ba) & (ba)^{\ast}=a^{\ast}\wedge(ba)^{+}=b^{+}\\
0 & \text{otherwise.}
\end{cases}
\]
\end{defn}

\begin{lem}
\label{lem:Category_with_partial_composition}Let $S$ be a reduced
$E$-Fountain semigroup which satisfies the generalized right ample
condition, then $\mathcal{C}(S)$ with composition as defined in \defref{Definition_of_composition}
is a category (with zero morphisms).
\end{lem}

\begin{proof}
Let $m$ be a morphism with $\rb(m)=e\in E$. If $m$ is a zero morphism
then clearly $C(e)\cdot m=m$. Otherwise $m=C(a)$ with $a^{+}=e$.
Since $(ea)^{\ast}=a^{\ast}$ and $(ea)^{+}=a^{+}=e$ we have 
\[
C(e)\cdot C(a)=C(ea)=C(a).
\]
Proving $m\cdot C(e)=m$ if $\db(m)=e$ is dual so morphisms of the
form $C(e)$ for $e\in E$ are identity morphisms. It is left to check
associativity. Let $m^{\prime\prime},m^{\prime},m$ be three morphisms
such that $\db(m^{\prime\prime})=\rb(m^{\prime})$ and $\db(m^{\prime})=\rb(m)$.
Clearly, if at least one of them is a zero morphism then 
\[
(m^{\prime\prime}\cdot m^{\prime})\cdot m=0=m^{\prime\prime}\cdot(m^{\prime}\cdot m).
\]
So it is left to handle the case $m^{\prime\prime}=C(c),\quad m^{\prime}=C(b),\quad m=C(a)$
where $a,b,c\in S$ are such that 
\[
c^{\ast}=b^{+},\quad b^{\ast}=a^{+}
\]
and we need to prove that 
\[
\left(C(c)\cdot C(b)\right)\cdot C(a)=C(c)\cdot\left(C(b)\cdot C(a)\right).
\]

If both sides are non-zero, it is clear that both equal $C(cba)$.
We need to show that the right hand side is non-zero if and only if
the left hand side is non-zero. First assume that the right hand side
is non-zero and equals $C(cba)$. This implies 
\[
(cba)^{\ast}=a^{\ast},\quad(cba)^{+}=c^{+}.
\]

We need to prove $C(c)\cdot C(b)=C(cb)$ and then $C(cb)\cdot C(a)=C(cba)$.
We know that $(cba)^{+}\leq(cb)^{+}\leq c^{+}$ but the assumption
is $(cba)^{+}=c^{+}$ so $(cb)^{+}=c^{+}$. The equality $(cba)^{\ast}=a^{\ast}$
implies by the generalized right ample condition that $(cb)^{\ast}=(cbb^{\ast})^{\ast}=(cba^{+})^{\ast}=a^{+}=b^{\ast}$
so $C(c)\cdot C(b)=C(cb)$. We know that $(cba)^{\ast}=a^{\ast}$,
$(cba)^{+}=c^{+}$ so $C(cb)\cdot C(a)=C(cba)$. This finishes this
side.

In the other direction assume that the left hand side is non-zero
and 
\[
\left(C(c)\cdot C(b)\right)\cdot C(a)=C(cba).
\]
This implies 
\[
(cba)^{\ast}=a^{\ast},\quad(cba)^{+}=c^{+}.
\]
We know that $(cba)^{\ast}\leq(ba)^{\ast}\leq a^{\ast}$ and $(cba)^{\ast}=a^{\ast}$
so this shows $(ba)^{\ast}=a^{\ast}$. According to \lemref{Ast_implies_plus}
we deduce that $(ba)^{+}=b^{+}$ as well. So $C(b)\cdot C(a)=C(ba)$.
We know that $(cba)^{\ast}=a^{\ast}$, $(cba)^{+}=c^{+}$ so $C(c)\cdot C(ba)=C(cba)$.
This finishes the proof.
\end{proof}
\begin{rem}
\label{rem:Simplify_the_category}Note that if the right congruence
identity holds, then composition of non-zero morphisms in $\mathcal{C}(S)$
coincide with the category composition from \corref{The_right_congruence_case}.
In this case, we can strip the zero morphisms from $\mathcal{C}(S)$
and obtain the category $\mathcal{C}_{\bullet}(S)$. In particular,
in this case $\Bbbk\mathcal{C}_{\bullet}(S)\simeq\Bbbk_{0}\mathcal{C}(S)$
for every commutative unital ring $\Bbbk$.
\end{rem}

We say that a relation $R$ on $S$ is \emph{principally finite} if
for every $a\in S$ the set $\{c\in S\mid c\,R\,a\}$ is finite. Define
a relation $\trianglelefteq_{l}$ on $S$ by the rule that $a\trianglelefteq_{l}b$
if and only if $a=be$ for some $e\in E$. It is proved in \cite[Lemma 3.5]{Stein2021}
that if $S$ is a reduced $E$-Fountain semigroup then $a\trianglelefteq_{l}b\iff a=ba^{\ast}$
and it is shown that the relation $\trianglelefteq_{l}$ is reflexive
but in general it is not anti-symmetric nor transitive. However, the
relation $\trianglelefteq_{l}$ is a partial order if $E$ is a commutative
subsemigroup of $S$ \cite[Section 3.2]{Gould2010b}.

Let $\Bbbk$ be a commutative unital ring. If we assume that $\mathcal{C}(S)$
is a category and the relation $\trianglelefteq_{l}$ is principally
finite we can define a function $\varphi:\Bbbk S\to\Bbbk_{0}\mathcal{C}(S)$
on basis elements by

\[
\varphi(a)=\sum_{c\trianglelefteq_{l}a}C(c).
\]

It is clear that $\varphi$ is a $\Bbbk$-module homomorphism.
\begin{thm}
Let $S$ be a reduced $E$-Fountain semigroup and assume that $\trianglelefteq_{l}$
is principally finite. Then, the generalized right ample condition
holds if and only if $\mathcal{C}(S)$ is a category and the $\Bbbk$-module
homomorphism $\varphi:\Bbbk S\to\Bbbk_{0}\mathcal{C}(S)$ is a homomorphism
of $\Bbbk$-algebras.
\end{thm}

\begin{proof}
First assume that the generalized right ample condition holds. Then
$\mathcal{C}(S)$ is a category by \lemref{Category_with_partial_composition}
and it is left to show that $\varphi$ is a homomorphism. For $a,b\in S$,
we need to show that $\varphi(ba)=\varphi(b)\varphi(a)$ so we need
to prove
\[
\sum_{c\trianglelefteq_{l}ba}C(c)=\sum_{c^{\prime\prime}\trianglelefteq_{l}b}C(c^{\prime\prime})\sum_{c^{\prime}\trianglelefteq_{l}a}C(c^{\prime}).
\]
First note that if $c^{\prime\prime}\trianglelefteq_{l}b$ and $c^{\prime}\trianglelefteq_{l}a$
and $C(c^{\prime\prime})C(c^{\prime})\neq0$ then $(c^{\prime\prime})^{\ast}=(c^{\prime})^{+}$
and so 
\[
c^{\prime\prime}c^{\prime}=b(c^{\prime\prime})^{\ast}c^{\prime}=b(c^{\prime})^{+}c^{\prime}=bc^{\prime}=ba(c^{\prime})^{\ast}
\]
so $c^{\prime\prime}c^{\prime}\trianglelefteq_{l}ba$. Therefore,
every element on the right-hand side appears also on the left-hand
side. Now, take $c\trianglelefteq_{l}ba$ and define $c^{\prime}=ac^{\ast}$
and $c^{\prime\prime}=b(ac^{\ast})^{+}$.

We observe that
\[
c^{\prime\prime}c^{\prime}=b(ac^{\ast})^{+}ac^{\ast}=bac^{\ast}=c.
\]

Next, note that 
\[
c^{\ast}=(bac^{\ast})^{\ast}\leq(ac^{\ast})^{\ast}\leq c^{\ast}
\]
so $(c^{\prime})^{\ast}=(ac^{\ast})^{\ast}=c^{\ast}=(c^{\prime\prime}c^{\prime})^{\ast}$.
Since $(bac^{\ast})^{\ast}=c^{\ast}=(ac^{\ast})^{\ast}$ the generalized
right ample condition implies

\[
(c^{\prime\prime})^{\ast}=(b(ac^{\ast})^{+})^{\ast}=(ac^{\ast})^{+}=(c^{\prime})^{+}.
\]
Finally, the fact that $(c^{\prime\prime}c^{\prime})^{+}=(c^{\prime\prime})^{+}$
follows from \lemref{Ast_implies_plus}. Therefore, the composition
$\mbox{\ensuremath{C(c^{\prime\prime})\cdot C(c^{\prime})}}$ is non-zero
and $C(c^{\prime\prime})\cdot C(c^{\prime})=C(c^{\prime\prime}c^{\prime})=C(c)$.
This shows that every element from the left-hand side appears on the
right-hand side. For this direction, it remains to show that it appears
only once. Assume $C(c)=C(d^{\prime\prime})C(d^{\prime})$ for $d^{\prime\prime}\trianglelefteq_{l}b$
and $d^{\prime}\trianglelefteq_{l}a$. Then $(d^{\prime})^{\ast}=c^{\ast}$
so we must have $d^{\prime}=a(d^{\prime})^{\ast}=ac^{\ast}=c^{\prime}$.
Since $C(d^{\prime\prime})\cdot C(d^{\prime})$ is non-zero we must
have $(d^{\prime\prime})^{\ast}=(d^{\prime})^{+}=(ac^{\ast})^{+}$
so $d^{\prime\prime}=b(d^{\prime\prime})^{\ast}=b(ac^{\ast})^{+}=c^{\prime\prime}$.
This proves uniqueness.

On the other direction, assume that $\mathcal{C}(S)$ is a category
and $\varphi$ is a homomorphism of $\Bbbk$-algebras. Let $a,b\in S$
and assume that $(ba)^{\ast}=a^{\ast}$. We need to prove that $(ba^{+})^{\ast}=a^{+}$.
The fact that $\varphi(ba)=\varphi(b)\varphi(a)$ implies
\[
\sum_{c\trianglelefteq_{l}ba}C(c)=\sum_{c^{\prime\prime}\trianglelefteq_{l}b}C(c^{\prime\prime})\sum_{c^{\prime}\trianglelefteq_{l}a}C(c^{\prime}).
\]
Choose $c=ba$ from the left-hand side. The morphism $C(ba)$ must
appear on the right-hand side. So there exists $c^{\prime}\trianglelefteq_{l}a$
and $c^{\prime\prime}\trianglelefteq_{l}b$ such that $C(ba)=C(c^{\prime\prime})C(c^{\prime})$.
This implies $(c^{\prime})^{\ast}=(ba)^{\ast}=a^{\ast}$ so $c^{\prime}=aa^{\ast}=a$.
Since the product $C(c)=C(c^{\prime\prime})C(a)$ is defined it must
be the case that $(c^{\prime\prime})^{\ast}=a^{+}$ and therefore
$c^{\prime\prime}=ba^{+}$. This implies $(ba^{+})^{\ast}=a^{+}$
as required.
\end{proof}
\begin{lem}
Let $S$ be a reduced $E$-Fountain semigroup and assume that $\mathcal{C}(S)$
is a category. If $\trianglelefteq_{l}\subseteq\preceq$ where $\preceq$
is some principally finite partial order, then $\varphi$ is an isomorphism
of $\Bbbk$-modules. 
\end{lem}

\begin{proof}
Identical to the proof given in \cite[Lemma 4.3]{Stein2021}.
\end{proof}
The immediate corollary is:
\begin{thm}
\label{thm:isomorphism_theorem} Let $S$ be a reduced $E$-Fountain
semigroup and assume that $\trianglelefteq_{l}$ is contained in a
principally finite partial order. Then, the generalized right ample
condition holds if and only if $\mathcal{C}(S)$ is a category and
$\varphi$ is an isomorphism of $\Bbbk$-algebras.
\end{thm}

\begin{rem}
\label{rem:Simplified_isomorphism}Assume the conditions of \thmref{isomorphism_theorem}
hold and $S$ satisfies also the right congruence condition. Then
$\mathcal{C}_{\bullet}(S)$ is a category and $\Bbbk\mathcal{C}_{\bullet}(S)\simeq\Bbbk_{0}\mathcal{C}(S)$
as noted in \remref{Simplify_the_category}. Therefore, in this case
we obtain a simplified isomorphism $\Bbbk S\simeq\Bbbk\mathcal{C}_{\bullet}(S)$
without a contracted algebra or zero morphisms.
\end{rem}

\section{\label{sec:Generalization}Generalization of older results}

In this section we will explain why \thmref{isomorphism_theorem}
is a simultaneous generalization of three different results in the
literature.

\paragraph{The congruence condition}

\thmref{isomorphism_theorem} is proved in \cite[Theorem 4.4]{Stein2021}
under the additional assumption that $S$ satisfies the congruence
condition. To be precise, \cite[Theorem 4.4]{Stein2021} is the simplified
case given in \remref{Simplified_isomorphism} because in this case
$\mathcal{C}_{\bullet}(S)$ is a category by \corref{The_right_congruence_case}.
This case includes inverse semigroups \cite{Steinberg2006}, semigroups
of partial functions \cite{Stein2016} and the Catalan monoid \cite{Margolis2018A}.

\paragraph{Generalized matrix representations and strict right ample semigroups}

Let $A_{1},\ldots,A_{n}$ be algebras over a commutative unital ring
$\Bbbk$. For every $1\leq i,j\leq n$ let $A_{i,j}$ be an $(A_{i},A_{j})$-bimodule
where $A_{i,i}=A_{i}$. Assume that for every triple $(i,j,k)$ there
exists an $(A_{i},A_{j})$-bimodule homomorphism $\Phi_{i,k}^{j}:A_{i,j}\otimes A_{j,k}\to A_{i,k}$
such that the associativity condition
\[
\Phi_{i,l}^{k}\circ(\Phi_{i,k}^{j}\otimes\id_{k,l})=\Phi_{i,l}^{j}\circ(\id_{i,j}\otimes\Phi_{j,l}^{k})
\]
holds (here $\id_{i,j}:A_{i,j}\to A_{i,j}$ is the identity function).
The $\Bbbk$-module $M[A_{i,j}]$ of all matrices
\[
\left(\begin{array}{cccc}
a_{1} & a_{1,2} & \cdots & a_{1,n}\\
a_{2,1} & a_{2}\\
\vdots &  & \ddots\\
a_{n,1} &  &  & a_{n}
\end{array}\right)
\]
with $a_{i}\in A_{i}$, $a_{i,j}\in A_{i,j}$ is a $\Bbbk$-algebra
with the obvious multiplication
\[
[M\cdot N]_{i,j}=\sum_{k=1}^{n}\Phi_{i,j}^{k}(M_{i,k}\otimes N_{k,j}).
\]
If $A$ is a $\Bbbk$-algebra, an isomorphism of $\Bbbk$-algebras
$\varphi:A\to M[A_{i,j}]$ is called a \emph{generalized matrix representation}
of $A$ with degree $n$. 

We can restate \thmref{isomorphism_theorem} in this language, at
least when $E$ is a finite set. Let $S$ be a reduced $E$-Fountain
semigroup which satisfies the generalized right ample condition and
assume that $E=\{e_{1},\ldots,e_{n}\}$ is a finite set. Define $A_{i,j}=\Bbbk_{0}\mathcal{C}(S)(e_{j},e_{i})$
where $\Bbbk$ is any commutative unital ring. Note that $A_{i}=\Bbbk_{0}\mathcal{C}(S)(e_{i},e_{i})$
is indeed a $\Bbbk$-algebra. Define $\Phi_{i,k}^{j}:A_{i,j}\otimes A_{j,k}\to A_{i,k}$
on basis elements by
\begin{align*}
\Phi_{i,k}^{j}(C(b)\otimes C(a)) & =C(b)\cdot C(a)
\end{align*}
The associativity of $\mathcal{C}(S)$ implies the associativity condition
on $M[A_{i,j}]$.

We can define $\mbox{\ensuremath{\varphi:\Bbbk_{0}\mathcal{C}(S)\to M[A_{i,j}]}}$
by setting
\[
[\varphi(C(s))]_{i,j}=\begin{cases}
C(s) & e_{i}=s^{+},\quad e_{j}=s^{\ast}\\
0 & \text{otherwise}.
\end{cases}
\]
It is clear that $\varphi$ is a generalized matrix representation
for $\Bbbk_{0}\mathcal{C}(S)$ of degree $|E|$. \thmref{isomorphism_theorem}
now implies:
\begin{thm}
\label{thm:Generalized_matrix_representation_format}Let $S$ be a
reduced $E$-Fountain semigroup and let $\Bbbk$ be a commutative
unital ring. Assume that $E$ is finite, $S$ satisfies the generalized
right ample condition and $\trianglelefteq_{l}$ is contained in a
partial order. Then $\Bbbk S$ has a generalized matrix representation
of degree $|E|$.
\end{thm}

This generalizes \cite[Theorem 3.2]{Guo2018} which was proved for
the class of strict right ample semigroups. A strict right ample semigroup
$S$ is a reduced $E$-Fountain semigroup (with $E=E(S)$) that satisfies
stronger assumptions than the assumptions of \thmref{Generalized_matrix_representation_format}.
For instance, in a strict right ample semigroup the set $E(S)$ is
a commutative subsemigroup so $ef=e\iff fe=e$ clearly holds and $\trianglelefteq_{l}$
is a partial order. Moreover, the right ample condition $ea=a(ea)^{\ast}$
holds and this is a stronger condition than the generalized right
ample condition (\cite[Proposition 3.14]{Stein2021}).

\paragraph{Finite idempotent semigroups with central idempotents}

Let $S$ be a finite idempotent semigroup with central idempotents.
This means that $S^{2}=S$ and $ea=ae$ for every $a\in S$ and $e\in E(S)$.
Setting $E=E(S)$, it is shown in \cite[Section 6]{Steinberg2022}
that every $s\in S$ has a left (right) identity from $E$. If $e,f\in E$
are right identities of $a$ then $ef$ is also an idempotent and
a right identity of $a$. Since $ef\leq e,f$ we deduce that every
$a\in S$ has a minimum right identity. Since $E$ is central it is
clear that $ef=e\iff fe=e$ and $e\in E$ is a right identity of $s$
if and only if it is left identity of $s$. Therefore, $S$ is a reduced
$E$-Fountain semigroup with $s^{+}=s^{\ast}$ for every $s\in S$.
\begin{lem}
Let $S$ be a reduced $E$-Fountain semigroup where the set $E$ is
central then $S$ satisfies the generalized right ample condition.
\end{lem}

\begin{proof}
Assume $(ba)^{\ast}=a^{\ast}$ for $a,b\in S$. We need to prove $(ba^{+})^{\ast}=a^{+}$
which is equivalent to $(ba^{\ast})^{\ast}=a^{\ast}$. Since $ba^{\ast}a^{\ast}=ba^{\ast}$
it is clear that $(ba^{\ast})^{\ast}\leq a^{\ast}$. On the other
hand
\[
ba(ba^{\ast})^{\ast}=ba^{\ast}a(ba^{\ast})^{\ast}=ba^{\ast}(ba^{\ast})^{\ast}a=ba^{\ast}a=ba
\]
so
\[
a^{\ast}=(ba)^{\ast}\leq(ba^{\ast})^{\ast}.
\]
Therefore $(ba^{\ast})^{\ast}=a^{\ast}$ as required.
\end{proof}
Since the set $E$ is commutative the relation $\trianglelefteq_{l}$
is a partial order. We have checked the requirements of \thmref{isomorphism_theorem}
so $\Bbbk S\simeq\Bbbk_{0}\mathcal{C}(S)$ for every commutative unital
ring $\Bbbk$. In this special case, $s^{+}=s^{\ast}$ for every $s\in S$
so every non-zero morphism in $\mathcal{C}(S)$ is an endomorphism.
In particular, the set of identity morphisms $\{C(e)\mid e\in E$\}
forms an orthogonal set of central idempotents of $\Bbbk_{0}\mathcal{C}(S)$
so 
\[
\Bbbk_{0}\mathcal{C}(S)\simeq\prod_{e\in E}\Bbbk_{0}\mathcal{C}(S)C(e)\simeq\prod_{e\in E}\Bbbk_{0}\mathcal{C}(S)(e,e)
\]
as $\Bbbk$-algebras. We deduce that:
\begin{thm}
\label{thm:Central_idempotents_theorem}Let $S$ be a finite semigroup
with central idempotents such that $S^{2}=S$ then
\[
\Bbbk S\simeq\prod_{e\in E}\Bbbk_{0}\mathcal{C}(S)(e,e).
\]
\end{thm}

\thmref{Central_idempotents_theorem} is precisely \cite[Theorem 6.5]{Steinberg2022}.

\section{\label{sec:Examples}Additional examples }

\subsection{The monoid of order-preserving functions}

Let $\T_{n}$ be the monoid of all functions $f:[n]\to[n]$ (where
$[n]=\{1,\ldots,n\}$). A function $f\in\T_{n}$ is called \emph{order-preserving
}if $i\leq j\implies f(i)\leq f(j)$ for every $i,j\in[n]$. We denote
by $\Op_{n}$ the submonoid of $\T_{n}$ of all order-preserving functions.
In this section we will see that $\Op_{n}$ is a reduced $E$-Fountain
semigroup that satisfies the generalized \emph{left} ample condition
so by the dual of \thmref{isomorphism_theorem} there is an isomorphism
$\Bbbk\Op_{n}\simeq\Bbbk_{0}\mathcal{C}(\Op_{n})$ of $\Bbbk$-algebras
for any commutative unital ring $\Bbbk$. We will also describe more
explicitly the composition in $\mathcal{C}(\Op_{n})$.

\paragraph*{Elementary observations}

Recall that the kernel of of a function $f\in\Op_{n}$ - $\ker(f)$
- is the equivalence relation on its domain defined by $(i_{1},i_{2})\in\ker(f)$
if and only if $f(i_{1})=f(i_{2})$. Since $f$ is order-preserving,
the kernel classes of $f$ are \emph{intervals }- if $i_{1}\leq i_{2}\leq i_{3}$
and $(i_{1},i_{3})\in\ker(f)$ then $(i_{1},i_{2})\in\ker f$ also.

Let $K_{1},\ldots,K_{l}$ be the kernel classes of $f$. Let $y_{i}=f(K_{i}$)
and assume that the indices are arranged such that $y_{1}<y_{2}<\cdots<y_{l}$.
Choose $x_{i}$ to be the maximal element of $K_{i}$ and note that
it must be the case that $x_{l}=n$. Now, set 
\[
X=\{x_{1},\ldots,x_{l}\},\quad Y=\{y_{1},\ldots,y_{l}\}
\]
so from every $f\in\Op_{n}$ we can extract two sets $X,Y\subseteq[n]$
such that $|X|=|Y|$ and $n\in X$. In the other direction, from two
such sets $X=\{x_{1},\ldots,x_{l}\}$ where $x_{l}=n$ and $\mbox{\ensuremath{Y=\{y_{1},\ldots,y_{l}\}}}$
we can retrieve an $f$ by setting 
\[
f(x)=\begin{cases}
y_{1} & x\leq x_{1}\\
y_{i} & x_{i-1}<x\leq x_{i}\quad(2\leq i\leq l).
\end{cases}
\]
As a conclusion, we observe that there is a one-to-one correspondence
between functions $f\in\Op_{n}$ and pairs of sets $X,Y\subseteq[n]$
such that $|X|=|Y|$ and $n\in X$. From now on we denote by $f_{Y,X}$
the function associated with $X,Y\subseteq[n]$ (in contrary to \cite[Section 4.2]{stein2022}
where this function was denoted $f_{X,Y})$. Note that $\im(f_{Y,X})=Y$
(where $\im(f)$ is the image of $f$) and $X$ contains the maximal
elements of the kernel classes of $f_{Y,X}$. Another easy observation
is that $f_{Z,Y}f_{Y,X}=f_{Z,X}$.

The monoid $\Op_{n}$ is a regular submonoid of $\T_{n}$ and its
Green's equivalence relations are well-known (see \cite[Section 2]{GomesHowie1992}).
In particular:
\begin{align*}
f_{Y_{1},X_{1}}\,\Lc\,f_{Y_{2},X_{2}} & \iff\ker(f_{Y_{1},X_{1}})=\ker(f_{Y_{2},X_{2}})\iff X_{1}=X_{2}\\
f_{Y_{1},X_{1}}\,\Rc\,f_{Y_{2},X_{2}} & \iff\im(f_{Y_{1},X_{1}})=\im(f_{Y_{2},X_{2}})\iff Y_{1}=Y_{2}
\end{align*}
We want to define our distinguished set of idempotents $E$. Note
that for every $X\subseteq[n]$ with $n\in X$ the function $f_{X,X}$
is an idempotent that can be described by $f_{X,X}(i)=\min\{x\in X\mid i\leq x\}$.
It will be convenient to define $e_{X}=f_{X,X}$ where $X\subseteq[n]$
with $n\in X$ and we set $E=\{e_{X}\mid n\in X\}$ (this is precisely
the set of idempotents of the Catalan monoid - see \cite[Section 17.5.2]{Steinberg2016}).
We have just seen that $f_{Y,X}\Lc e_{X}$ and since $\Lc\subseteq\Lt$
we have $f_{Y,X}\,\Lt\,e_{X}$. On the other hand, it is not always
true that $f_{Y,X}\,\Rc\,e_{Y}$ because $e_{Y}$ is defined only
if $n\in Y$. Denote by $\OF_{n}$ the submonoid of $\Op_{n}$ of
all order-preserving functions such that $f(n)=n$ and note that $\OF_{n}=\{f_{Y,X}\in\Op_{n}\mid n\in Y\}$.
This monoid was discussed in detail in \cite[Section 4.2]{stein2022}.
We have $f_{Y,X}\,\Rc\,e_{Y}$ if only if $f_{Y,X}\in\OF_{n}$. However,
it is easy to see that 
\[
e_{Z}f_{Y,X}=f_{Y,X}
\]
if and only if $Y\subseteq Z$. Therefore, the minimal left identity
of $f_{Y,X}$ from $E$ is $e_{Y\cup\{n\}}$ hence $f_{Y,X}\,\Rt\,e_{Y\cup\{n\}}$.
It will be convenient to set $\widetilde{Y}=Y\cup\{n\}$ for every
$Y\subseteq[n]$ so $f_{Y,X}\,\Rt\,e_{\widetilde{Y}}$ for every $f_{Y,X}\in\Op_{n}$.
The first conclusion is that there exists an idempotent of $E$ in
every $\Lt$-class and $\Rt$-class of $\Op_{n}$. It is also easy
to check that $e_{Y}e_{X}=e_{X}\iff e_{X}e_{Y}=e_{X}\iff X\subseteq Y$
(see also \cite[Section 6]{Stein2021} where we derived this fact
from \cite[Lemma 3.6]{Denton2010}) so $\Op_{n}$ is a reduced $E$-Fountain
semigroup. Since $\Lc\subseteq\Lt$ and there is a unique element
of $E$ in every $\Lt$-class we deduce that $\Lc=\Lt$ and therefore
$\Lt$ is a right congruence so $\Op_{n}$ satisfies the right congruence
condition. We will see later that $(f_{W,Z})^{\ast}=(f_{Y,X})^{+}$
does not always imply that $(f_{W,Z}f_{Y,X})^{+}=(f_{W,Z})^{+}$ so
in view of \corref{The_right_congruence_case} we deduce that $\Op_{n}$
does not satisfy the generalized right ample condition. However, it
turns out that it satisfies the generalized left ample condition.

\paragraph*{Generalized left ample condition and isomorphism of algebras}

First, we state the dual of the relevant results from \secref{Isomorphism}.
Every $\Rt$-class of a reduced $E$-Fountain semigroup $S$ is an
incomplete right $S$-action and a partial function
\[
l_{a}:\Rt(a^{\ast})\to\Rt(a^{+})
\]
is defined by
\[
l_{a}(b)=\begin{cases}
ab & (ab)^{+}=a^{+}\\
\text{undefined} & \text{otherwise}.
\end{cases}
\]
 The dual of \propref{Equivalent_characterizations_of_gra} is:
\begin{prop}
\label{prop:Equivalent_characterizations_of_gra_dual}Let $S$ be
a reduced $E$-Fountain semigroup. The following conditions are equivalent.
\begin{enumerate}
\item The semigroup $S$ satisfies $(((ab)^{+}a)^{\ast}b)^{+}=((ab)^{+}a)^{\ast}$
for every $a,b\in S$.
\item The semigroup $S$ satisfies $(ab)^{+}=a^{+}\implies(a^{\ast}b)^{+}=a^{\ast}$
for every $a,b\in S$.
\item The partial function $l_{a}:\Rt(a^{\ast})\to\Rt(a^{+})$ is a partial
homomorphism of incomplete right $S$-actions for every $a\in S$.
\end{enumerate}
\end{prop}

Define a relation $\trianglelefteq_{r}$ on $S$ by the rule that
$a\trianglelefteq_{r}b$ if and only if $a=eb$ for some $e\in E$.
If we assume that $\mathcal{C}(S)$ is a category and the relation
$\trianglelefteq_{r}$ is principally finite we can define a function
$\varphi:\Bbbk S\to\Bbbk_{0}\mathcal{C}(S)$ on basis elements by 

\[
\varphi(a)=\sum_{c\trianglelefteq_{r}a}C(c).
\]
The dual of \thmref{isomorphism_theorem} is:
\begin{thm}
\label{thm:isomorphism_theorem-dual} Let $S$ be a reduced $E$-Fountain
semigroup and assume that $\trianglelefteq_{r}$ is contained in a
principally finite partial order. Then, the generalized left ample
condition holds if and only if $\mathcal{C}(S)$ is a category and
$\varphi$ is an isomorphism of $\Bbbk$-algebras.
\end{thm}

Next, we want to prove that $\Op_{n}$ satisfies the generalized left
ample condition.
\begin{lem}
\label{lem:Image_lemma}Let $f,g\in\Op_{n}$. If $\im(fg)=\im(f)$
then $(f^{\ast}g)^{+}=f^{\ast}$.
\end{lem}

\begin{proof}
First note that $\im(fg)=\im(f)$ if and only if $\im(g)$ intersects
every kernel class of $f$. Now, $\ker(f)=\ker(f^{\ast})$ so $\im(fg)=\im(f)$
implies that $g$ intersects every kernel class of $f^{\ast}$ and
therefore $\im(f^{\ast}g)=\im(f^{\ast})$. This implies 
\[
(f^{\ast}g)^{+}=f^{\ast}.
\]
\end{proof}
\begin{prop}
\label{prop:lef_ample_on}The monoid $\Op_{n}$ satisfies the generalized
left ample condition.
\end{prop}

\begin{proof}
Let $f,g\in\Op_{n}$ and assume $(fg)^{+}=f^{+}$. We need to prove
$(f^{\ast}g)^{+}=f^{\ast}$. The fact that $(fg)^{+}=f^{+}$ implies
that $\im(fg)\cup\{n\}=\im(f)\cup\{n\}$. Since $\im(fg)\subseteq\im(f)$
we have two options. If $\im(fg)=\im(f)$ then $(f^{\ast}g)^{+}=f^{\ast}$
by \lemref{Image_lemma}. The only other option is 
\[
\im(fg)=\im(f)\backslash\{n\}.
\]
In this case $\im(g)$ intersects all the kernel classes of $f$ except
the kernel class of $n$. Since $\ker(f)=\ker(f^{\ast})$, the image
$\im(g)$ intersects all the kernel classes of $f^{\ast}$ except
the kernel class of $n$. Therefore 
\[
\im(f^{\ast}g)=\im(f^{\ast})\backslash\{n\}
\]
which also implies $(f^{\ast}g)^{+}=f^{\ast}$ as required.
\end{proof}
We want to prove that $\trianglelefteq_{r}$ on $\Op_{n}$ is contained
in a partial order. The following lemma will be useful.
\begin{lem}
\label{lem:Maximal_kenels_contained}Let $f_{Y,X},f_{W,Z}\in\Op_{n}$.
If $f_{Y,X}\trianglelefteq_{r}f_{W,Z}$ then $X\subseteq Z$.
\end{lem}

\begin{proof}
The assumption that $f_{Y,X}\trianglelefteq_{r}f_{W,Z}$ implies that
\[
f_{Y,X}=ef_{W,Z}
\]
 for some idempotent $e\in E$. Recall that $X$ is the set of all
maximal elements of kernel classes of $f_{Y,X}$ and note that $x\neq n$
is a maximal element of a kernel class of $f\in\Op_{n}$ if and only
if $f(x)\neq f(x+1)$. If $f_{W,Z}(x)=f_{W,Z}(x+1)$ then 
\[
f_{Y,X}(x)=ef_{W,Z}(x)=ef_{W,Z}(x+1)=f_{Y,X}(x+1).
\]
Therefore if $x\neq n$ is a maximal element of a kernel class of
$f_{Y,X}$ then it is a maximal element of a kernel class of $f_{W,Z}$.
With the fact that $n\in X,Z$ we deduce that $X\subseteq Z$.
\end{proof}
Let $Y=\{y_{1},\ldots,y_{k}\}$, $W=\{w_{1},\ldots,w_{k}\}$ (arranged
in increasing order) be two subsets of $[n]$ of size $k$. We write
$Y\geq W$ if $y_{i}\geq w_{i}$ for every $i$.
\begin{lem}
\label{lem:Contained_in_principally_finite}The relation $\trianglelefteq_{r}$
on $\Op_{n}$ is contained in a partial order.
\end{lem}

\begin{proof}
Define a relation $\preceq$ on $\Op_{n}$ by $f_{Y,X}\preceq f_{W,Z}$
if $X\subsetneqq Z$ or ($X=Z$ and $Y\geq W$) . It is easy to verify
that $\preceq$ is indeed a partial order on $\Op_{n}$ (it is the
lexicographic order on the Cartesian product of $\subseteq$ and $\leq$).
We prove that $\trianglelefteq_{r}$ is contained in $\preceq$. If
$f_{Y,X},f_{W,Z}\in\Op_{n}$ such that $f_{Y,X}\trianglelefteq_{r}f_{W,Z}$
then $X\subseteq Z$ by \lemref{Maximal_kenels_contained}. If $X\subsetneqq Z$
then we are done. If $X=Z$ then $|Y|=|X|=|Z|=|W|$, say 
\[
Z=X=\{x_{1},\ldots,x_{k}\},\quad Y=\{y_{1},\ldots,y_{k}\},\quad W=\{w_{1},\ldots,w_{k}\}
\]
Since $f_{Y,X}\trianglelefteq_{r}f_{W,Z}$ then $f_{Y,X}=ef_{W,X}$
for some idempotent $e\in E$. Now 
\[
y_{i}=f_{Y,X}(x_{i})=ef_{W,X}(x_{i})\geq f_{W,X}(x_{i})=w_{i}
\]
so indeed $Y\geq W$ and this finishes the proof\@.
\end{proof}
According to \thmref{isomorphism_theorem-dual}, \propref{lef_ample_on}
and \lemref{Contained_in_principally_finite} we obtain:
\begin{thm}
\label{thm:Order_preserving_isomorphism}The graph $\mathcal{C}(\Op_{n})$
with composition as defined in \defref{Definition_of_composition}
is a category (with zero morphisms) and for every commutative unital
ring $\Bbbk$ there is an isomorphism of algebras $\Bbbk\Op_{n}\simeq\Bbbk_{0}\mathcal{C}(\Op_{n})$.
\end{thm}

\begin{rem}
The fact that $\Bbbk\Op_{n}$ and $\Bbbk_{0}\mathcal{C}(\Op_{n})$
are Morita equivalent (i.e., their module categories are equivalent)
follows from the well-known Dold-Kan correspondence \cite{Dold1958,Kan1958}.
It is possible to show that $\Bbbk\Op_{n}$ is Morita equivalent to
the algebra of the category of sets $\{1,\ldots,k\}$ for $1\leq k\leq n$
and all the order-preserving functions between them - this is the
simplex category \cite[Section VII 5]{Mac_Lane1998} truncated at
$n$. According to the Dold-Kan theorem, this algebra is Morita equivalent
to the contracted algebra of the skeleton of $\mathcal{C}(\Op_{n})$
which is Morita equivalent to $\Bbbk_{0}\mathcal{C}(\Op_{n})$ itself
(see \cite[Proposition 2.2]{Webb2007}). The contracted algebra of
the skeleton is the algebra of a straight line quiver with $n$ vertices
and the relations are that each composition of successive arrows equals
$0$. See also the Lack-Street equivalence \cite{Lack2015} and \cite[Example 4.9]{Street2024}.
\end{rem}

\paragraph*{Associated category}

We want to show now that we can give a simple description of the composition
in $\mathcal{C}(\Op_{n})$. We start with some useful lemmas.
\begin{lem}
\label{lem:Composition_in_graph1}Let $f_{W,Z},f_{Y,X}\in\Op_{n}$.
If $(f_{W,Z})^{\ast}=(f_{Y,X})^{+}$ and $f_{Y,X}\in\OF_{n}$ then
$f_{W,Z}f_{Y,X}=f_{W,X}$.
\end{lem}

\begin{proof}
If $(f_{W,Z})^{\ast}=(f_{Y,X})^{+}$ then $Z=\widetilde{Y}$ (recall
that $\widetilde{Y}=Y\cup\{n\}$) but $f_{Y,X}\in\OF_{n}$ implies
$n\in Y$ so $\widetilde{Y}=Y$. Therefore $Z=Y$ and $f_{W,Z}f_{Y,X}=f_{W,X}$.
\end{proof}
For every non-empty $W\subseteq[n]$ we set $\max(W)$ to be the maximal
element of $W$ and $\mbox{\ensuremath{W^{\prime}=W\backslash\{\max(W)\}}}$.
\begin{lem}
\label{lem:Composition_in_graph2}Let $f_{W,Z},f_{Y,X}\in\Op_{n}$.
If $(f_{W,Z})^{\ast}=(f_{Y,X})^{+}$ and $f_{Y,X}\notin\OF_{n}$ then
$f_{W,Z}f_{Y,X}=f_{W^{\prime},X}$.
\end{lem}

\begin{proof}
If $(f_{W,Z})^{\ast}=(f_{Y,X})^{+}$ then $Z=\widetilde{Y}$ but $f_{Y,X}\notin\OF_{n}$
implies $n\notin Y$ so $f_{W,Z}(n)=\max(W)$ is not in the image
of $f_{W,Z}f_{Y,X}$. Every $w\in W^{\prime}$ is the image of some
$z\in Z\backslash\{n\}=Y$ and therefore $w$ is in the image of $f_{W,Z}f_{Y,X}$.
It is now clear that $f_{W,Z}f_{Y,X}=f_{W^{\prime},X}$.
\end{proof}
\begin{lem}
\label{lem:Closedness_of_good_arrows}Let $f_{W,Z},f_{Y,X}\in\Op_{n}$
such that $(f_{W,Z})^{\ast}=(f_{Y,X})^{+}$. If $f_{W,Z}\in\OF_{n}$
and $f_{Y,X}\notin\OF_{n}$ or vice versa then $f_{W,Z}f_{Y,X}\notin\OF_{n}$.
\end{lem}

\begin{proof}
If $f_{W,Z}\in\OF_{n}$ and $f_{Y,X}\notin\OF_{n}$ then $\max(W)=n$
and $W^{\prime}=W\backslash\{n\}$. By \lemref{Composition_in_graph2}
we deduce $f_{W,Z}f_{Y,X}=f_{W\backslash\{n\},X}\notin\OF_{n}$. If
$f_{W,Z}\notin\OF_{n}$ then the image of $f_{W,Z}f_{Y,X}$ which
is contained in the image of $f_{W,Z}$ does not contain $n$ so $f_{W,Z}f_{Y,X}\notin\OF_{n}.$
\end{proof}
\begin{lem}
\label{lem:Almost_left_congruence}Let $f_{W,Z},f_{Y,X}\in\Op_{n}$.
If $(f_{W,Z})^{\ast}=(f_{Y,X})^{+}$ and at least one of $f_{W,Z},f_{Y,X}$
is in $\OF_{n}$ then $(f_{W,Z}f_{Y,X})^{+}=(f_{W,Z})^{+}$. If $f_{W,Z},f_{Y,X}\notin\OF_{n}$
then $(f_{W,Z}f_{Y,X})^{+}\neq(f_{W,Z})^{+}$.
\end{lem}

\begin{proof}
If $f_{Y,X}\in\OF_{n}$ then \lemref{Composition_in_graph1} implies
$f_{W,Z}f_{Y,X}=f_{W,X}$ so $(f_{W,Z}f_{Y,X})^{+}=e_{\widetilde{W}}=(f_{W,Z})^{+}$.
If $f_{Y,X}\notin\OF_{n}$ then $f_{W,Z}f_{Y,X}=f_{W^{\prime},X}$
by \lemref{Composition_in_graph2} so $(f_{W,Z}f_{Y,X})^{+}=e_{\widetilde{W^{\prime}}}$
where $(f_{W,Z})^{+}=e_{\widetilde{W}}$. If $f_{W,Z}\in\OF_{n}$
then $\max(W)=n$ and $W^{\prime}=W\backslash\{n\}$ so $\widetilde{W^{\prime}}=W=\widetilde{W}$
and we obtain $(f_{W,Z}f_{Y,X})^{+}=(f_{W,Z})^{+}$. If $f_{W,Z},f_{Y,X}\notin\OF_{n}$
then $\max(W)\neq n$ so $\max(W)\notin\widetilde{W^{\prime}}$ hence
$\widetilde{W^{\prime}}\neq\widetilde{W}$ and $\mbox{\ensuremath{(f_{W,Z}f_{Y,X})^{+}\neq(f_{W,Z})^{+}}}$. 
\end{proof}
The graph $\mathcal{C}(\Op_{n})$ associated with $\Op_{n}$ can be
described as follows. The objects are in one-to-one correspondence
with the set of idempotents $E=\{e_{X}\}$ (where $X\subseteq[n]$
such that $n\in X$). For every $f_{Y,X}\in\Op_{n}$ we associate
a non-zero morphism $C(f_{Y,X})$ from $e_{X}$ to $e_{\widetilde{Y}}$.
Before we turn to describe the composition we emphasize an important
observation. Let $X,Y\subseteq[n]$ such that $n\in X,Y$. If $|Y|=|X|$
then $C(f_{Y,X})$ is the unique non-zero morphism from $e_{X}$ to
$e_{Y}$ (note that $f_{Y,X}\in\OF_{n}$ in this case). If $|Y|=|X|+1$
then $C(f_{Y^{\prime},X})$ is the unique non-zero morphism from $e_{X}$
to $e_{Y}$ (note that $Y^{\prime}=Y\backslash\{n\}$ and $f_{Y^{\prime},X}\notin\OF_{n}$
in this case). If $|Y|\notin\{|X|,|X|+1\}$ then there are no non-zero
morphisms from $e_{X}$ to $e_{Y}$. 
\begin{example}
The graph $\mathcal{C}(\Op_{3})$ has the four objects $e_{\{3\}},e_{\{1,3\}},e_{\{2,3\}},e_{\{1,2,3\}}$
and $10$ non-zero morphisms. If we naturally draw $e_{Y}$ above
$e_{X}$ if $|Y|>|X|$ and ignore the identity and zero morphisms
then the graph $\mathcal{C}(\Op_{3})$ is the following graph:

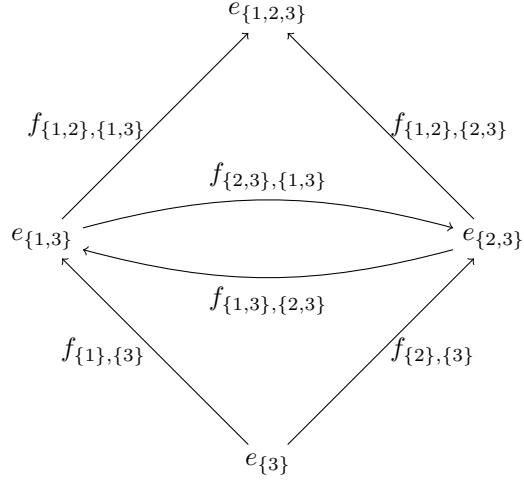
\begin{center} \begin{tikzpicture}\path (0,6) node (4) {$e_{\{1,2,3\}}$};  \path (-3,3) node (2) {$e_{\{1,3\}}$}; \path (3,3) node (3) {$e_{\{2,3\}}$}; \path (0,0) node (1) {$e_{\{3\}}$};   \draw[->] (2) to[out=15, in =165] node[midway,above]{$f_{\{2,3\},\{1,3\}}$}(3); \draw[->] (3) to[out=195, in =345] node[midway,below]{$f_{\{1,3\},\{2,3\}}$}(2); \draw[->] (2) to node[midway,left]{$f_{\{1,2\},\{1,3\}}$}(4); \draw[->] (3) to node[midway,right]{$f_{\{1,2\},\{2,3\}}$}(4); \draw[->] (1) to node[midway,left]{$f_{\{1\},\{3\}}$}(2); \draw[->] (1) to node[midway,right]{$f_{\{2\},\{3\}}$}(3);\end{tikzpicture}\end{center}
\end{example}

Recall that composition in $\mathcal{C}(\Op_{n})$ of two non-zero
morphisms $C(g),C(f)$ with $g^{\ast}=f^{+}$ is defined by: 
\[
C(g)\cdot C(f)=\begin{cases}
C(gf) & (gf)^{\ast}=f^{\ast},\quad(gf)^{+}=g^{+}\\
0 & \text{otherwise}.
\end{cases}
\]
We can simplify this description. Since $\Op_{n}$ satisfies the right
congruence condition, $g^{\ast}=f^{+}\implies(gf)^{\ast}=f^{\ast}$
so we can drop the requirement for $(gf)^{\ast}=f^{\ast}$. From \lemref{Almost_left_congruence}
we deduce that $g^{\ast}=f^{+}$ implies $(gf)^{+}=g^{+}$ if at least
one of $g,f$ is in $\OF_{n}$ and $(gf)^{+}\neq g^{+}$ otherwise.
Therefore, we obtain:
\begin{lem}
\label{lem:Partial_composition_description}Composition in $\mathcal{C}(\Op_{n})$
of two non-zero morphisms $C(g),C(f)$ with $g^{\ast}=f^{+}$ can
be described by 
\[
C(g)\cdot C(f)=\begin{cases}
C(gf) & f\in\OF_{n}\text{ or }g\in\OF_{n}\\
0 & \text{otherwise}.
\end{cases}
\]
\end{lem}

\paragraph*{Maximal semisimple image}

We finish this subsection with an observation on the maximal semisimple
image of the algebra $\Bbbk\Op_{n}$ where $\Bbbk$ is a field. Let
$D_{n}$ be the subcategory of $\mathcal{C}(\Op_{n})$ with the same
set of objects but the morphisms are only morphisms of the form $C(f_{Y,X})$
where $f_{Y,X}\in\OF_{n}$. Note that if $f_{W,Z},f_{Y,X}\in\OF_{n}$
then $\widetilde{Y}=Y$ (since $n\in Y$) and if $Z=Y$ then $C(f_{W,Z})\cdot C(f_{Y,X})=C(f_{W,X})$
where $f_{W,X}\in\OF_{n}$. Therefore, $D_{n}$ is indeed a subcategory
(without zero morphisms) of $\mathcal{C}(\Op_{n})$. Moreover, every
morphism $C(f_{Y,X})$ in $D_{n}$ has an inverse $C(f_{X,Y})$ so
$D_{n}$ is a groupoid. Note also that the identity morphisms $C(e_{X})$
(for $X\subseteq[n]$ with $n\in X$) are the only endomorphisms in
$D_{n}$.
\begin{lem}
Let $\Bbbk$ be a field. The Jacobson radical of $\Bbbk_{0}\mathcal{C}(\Op_{n})$
is the subvector space spanned by all the morphisms of the form $C(f_{Y,X})$
where $f_{Y,X}\notin\OF_{n}$ (or equivalently, $n\notin Y)$ and
\[
\Bbbk_{0}\mathcal{C}(\Op_{n})/\Rad(\Bbbk\mathcal{C}(\Op_{n}))\simeq\Bbbk D_{n}.
\]
\end{lem}

\begin{proof}
Denote by $R$ the subvector space of $\Bbbk_{0}\mathcal{C}(\Op_{n})$
spanned by all the morphisms of the form $C(f_{Y,X})$ where $f_{Y,X}\notin\OF_{n}$.
First note that according to \lemref{Partial_composition_description}
if $f_{W,Z},f_{Y,X}\notin\OF_{n}$ then 
\[
C(f_{W,Z})\cdot C(f_{Y,X})=0.
\]
Moreover, if $f_{Y,X}\notin\OF_{n}$ and $f_{W,Z}\in\OF_{n}$ are
such that $C(f_{W,Z})\cdot C(f_{Y,X})\neq0$ then $f_{W,Z}f_{Y,X}\notin\OF_{n}$
according to \lemref{Closedness_of_good_arrows} (and the dual holds
if $C(f_{Y,X})\cdot C(f_{W,Z})\neq0$). Both facts together imply
that $R$ is a nilpotent ideal (of nilpotency degree $2$). Since
every non-zero morphism of $\mathcal{C}(\Op_{n})$ is either in $R$
or in $D_{n}$ it is clear that $\Bbbk_{0}\mathcal{C}(\Op_{n})/R\simeq\Bbbk D_{n}$.
It is well-known that the algebra $\Bbbk\mathcal{G}$ of a finite
groupoid $\mathcal{G}$ is semisimple if and only if the order of
every endomorphism group is invertible in the field $\Bbbk$ (\cite[Theorem 8.15]{Steinberg2016}).
In our case, the endomorphism groups of $D_{n}$ are trivial so $\Bbbk D_{n}$
is semisimple for every field $\Bbbk$. In conclusion, $R$ is a nilpotent
ideal such that $\Bbbk_{0}\mathcal{C}(\Op_{n})/R$ is semisimple so
$R=\Rad\Bbbk_{0}\mathcal{C}(\Op_{n})$ and the statement follows.
\end{proof}
\begin{cor}
For any field $\Bbbk$. The maximal semisimple image of $\Bbbk\Op_{n}$
is isomorphic to $\Bbbk\OF_{n}$.
\end{cor}

\begin{proof}
According to \thmref{Order_preserving_isomorphism}, $\Bbbk\Op_{n}\simeq\Bbbk_{0}\mathcal{C}(\Op_{n})$.
We have just proved that the maximal semisimple image of the last
is $\Bbbk D_{n}$. Finally, it is shown in \cite[Section 4.2]{stein2022}
that $\Bbbk D_{n}\simeq\Bbbk\OF_{n}$ which finishes the proof.
\end{proof}

\subsection{Binary relations with demonic composition}

Let $\B_{n}$ be the set of binary relations on the set $[n]=\{1,\ldots,n\}$.
It is a monoid under the standard ``angelic'' composition:
\[
\beta\cdot\alpha=\{(x,y)\mid\exists z\quad(x,z)\in\alpha,\quad(z,y)\in\beta\}
\]
-note that we are composing right to left here. Set $E=\{\id_{A}\mid A\subseteq[n]\}$
where $\id_{A}=\{(a,a)\mid a\in A\}$. This is a commutative subsemigroup
of $\B_{n}$. The monoid $\B_{n}$ is a reduced $E$-Fountain semigroup
which satisfies both the right and left congruence conditions and
\[
\alpha^{\ast}=\id_{\dom(\alpha)},\quad\alpha^{+}=\id_{\im(\alpha)}
\]
(see \cite[Section 3.2]{East2021}). We can also define ``demonic''
composition:

\[
\beta\ast\alpha=\{(x,y)\in\beta\cdot\alpha\mid(x,z)\in\alpha\implies z\in\dom(\beta)\}.
\]
Denote by $\B_{n}^{d}$ the set $\B_{n}$ with demonic composition
as operation. It is known that $\B_{n}^{d}$ is indeed a monoid. Moreover,
it is a reduced $E$-Fountain semigroup which satisfies the right
congruence condition and the right ample condition (see \cite{Stokes2021}).
The right ample condition is a stronger condition than our generalized
right ample condition (\cite[Proposition 3.14]{Stein2021}). Since
$E$ is a commutative subsemigroup we know that $\trianglelefteq_{l}$
is a partial order. Set $\mathcal{C}=\mathcal{C}_{\bullet}(\B_{n}^{d})$
and note that here $\mathcal{C}$ is a category by \corref{The_right_congruence_case}.
Therefore, we can apply \remref{Simplified_isomorphism}. The objects
of $\mathcal{C}$ can be identified with subsets of $[n]$. For every
$\alpha\in\B_{n}$ there is an associated morphism $C(\alpha)$ from
$\dom(\alpha)$ to $\im(\alpha)$. If $\beta^{\ast}=\alpha^{+}$ then
$\dom(\beta)=\im(\alpha)$ hence $\beta\ast\alpha=\beta\cdot\alpha$.
This shows that for composition of morphisms in this category, the
angelic and demonic definitions coincide. \remref{Simplified_isomorphism}
implies:
\begin{thm}
Let $\Bbbk$ be a commutative unital ring. Let $\mathcal{C}$ be the
category whose objects are subsets of $[n]$ and the hom-set $\mathcal{C}(X,Y)$
(for $X,Y\subseteq[n]$) contains all total onto relations from $X$
to $Y$. Then 
\[
\Bbbk\B_{n}^{d}\simeq\Bbbk\mathcal{C}
\]
is an isomorphism of $\Bbbk$-algebras.
\end{thm}

\begin{rem}
Similar examples can be derived from any domain semiring as explained
in \cite[Section 6.1]{Stokes2021}.
\end{rem}

For every relation $\alpha\in\B_{n}$ denote by $\alpha^{-1}$ the
converse relation $\alpha^{-1}=\{(y,x)\mid(x,y)\in\alpha\}$. It is
clear that we can define $(-)^{-1}:\mathcal{C}\to\mathcal{C}$ a contravariant
isomorphism which is identity on objects and $(C(\alpha))^{-1}=C(\alpha^{-1})$
on morphisms. Therefore, $\mathcal{C\simeq\mathcal{C^{\op}}}$ and
we can deduce that:
\begin{prop}
The algebra $\Bbbk\B_{n}^{d}$ is isomorphic to its opposite.
\end{prop}

\bibliographystyle{plain}
\bibliography{library}

\end{document}